\documentclass[11pt,leqno]{amsart}
\usepackage{amssymb,amsthm,amsmath}
\usepackage{graphicx}

\usepackage{enumerate}

\DeclareMathOperator{\ob}{ob}
\DeclareMathOperator{\mor}{mor}

\DeclareMathOperator{\Aut}{Aut}
\DeclareMathOperator{\Endom}{End}
\DeclareMathOperator{\Hom}{Hom}
\DeclareMathOperator{\Grd}{Grd}
\newcommand{\C}{\mathbb{C}}

\theoremstyle{plain}
\newtheorem{theorem}{Theorem}

\newtheorem{proposition}{Proposition}

\theoremstyle{definition}
\newtheorem{definition}{Definition}
\newtheorem{example}{Example}
\newtheorem{remark}{Remark}

\begin{document}

\title[Miyashita Action in Strongly Groupoid Graded Rings]{Miyashita Action in Strongly Groupoid Graded Rings}


\author[J. \"{O}inert]{Johan \"{O}inert}
\address{Johan \"{O}inert, Department of Mathematical Sciences\\
University of Copenhagen\\
Universitetsparken 5\\
DK-2100 Copenhagen \O\\
Denmark}
\email{oinert@math.ku.dk}
\urladdr{http://www.math.ku.dk/$\sim$oinert/}
\thanks{The first author was partially supported by The Swedish Research Council, The Swedish Foundation for International Cooperation in Research and Higher Education (STINT), The Crafoord Foundation, The Royal Physiographic Society in Lund, The Swedish Royal Academy of Sciences and "LieGrits", a Marie Curie Research Training Network funded by the European Community as project MRTN-CT 2003-505078.}

\author[P. Lundstr\"{o}m]{Patrik Lundstr\"{o}m}
\address{Patrik Lundstr\"{o}m, University West\\
Department of Engineering Science\\
SE-46186 Trollh\"{a}ttan\\
Sweden}
\email{patrik.lundstrom@hv.se}

\begin{abstract}
We determine the commutant of homogeneous subrings
in strongly groupoid
graded rings in terms of an action on the ring
induced by the grading.
Thereby we generalize a classical result of Miyashita from the group
graded case to the groupoid graded situation.
In the end of the article we exemplify this result.
To this end, we show, by an explicit construction,
that given a finite groupoid $G$, equipped with a nonidentity
morphism $t : d(t) \rightarrow c(t)$,
there is a strongly $G$-graded ring $R$
with the properties that each $R_s$, for $s \in G$,
is nonzero and $R_t$ is a nonfree left $R_{c(t)}$-module.
\end{abstract}

\maketitle

\section{Introduction}

Let $R$ be a ring.
By this we always mean that $R$ is an additive group
equipped with a multiplication which is associative.
If $R$ is unital, then the identity element of $R$ is denoted $1_R$.
We say that a subset $R'$ of $R$ is a subring of $R$
if it is itself a ring under the binary operations of $R$;
note that even if $R$ and $R'$ are unital it may happen that $1_{R'} \neq 1_R$.
However, we always assume that ring homomorphisms
$R \rightarrow R''$ between unital rings $R$ and $R''$ map $1_R$ to $1_{R''}$.
The group of ring automorphisms of $R$ is denoted $\Aut(R)$.

By the \emph{commutant}
of a subset $X$ of $R$, denoted $C_R(X)$, we mean the set of elements
of $R$ that commute with each element of $X$.
If $Y$ is another subset of $R$, then
$XY$ denotes the set of all finite sums of
products $xy$, for $x \in X$ and $y \in Y$.
The task of calculating $C_R(X)$ is in general
a difficult problem.
However, if $R$ is strongly group graded
and $X$ belongs to a certain class of subrings of $R$,
then, by a classical result of Miyashita \cite{miy}
(see Theorem \ref{commutanttheorem}), there is
an elegant solution to this problem formulated in terms
of a group action defined by the grading.
Namely, recall that $R$ is said to be \emph{graded} by the group $G$,
or $G$-graded, if there is a set of additive subgroups,
$R_s$, for $s \in G$, of $R$ such that $R =
\bigoplus_{s \in G} R_s$ and $R_s R_t \subseteq
R_{st}$, for $s,t \in G$.
If $H$ is a subgroup of $G$, then we let
$R_H$ denote the subring $\bigoplus_{s \in H} R_s$ of $R$;
in particular, $R_e$ is a subring of $R$,
where $e$ denotes the identity element of $G$.
If $R$ is graded by $G$ and $R_s R_t = R_{st}$, for $s,t \in G$,
then $R$ is said to be \emph{strongly graded}.
If in addition $R$ is unital, then there is a unique group action
$G \ni s \mapsto \sigma_s \in {\rm Aut}(C_R(R_e))$
of $G$ on $C_R(R_e)$ satisfying
$r_s x = \sigma_s(x) r_s$, for $s \in G$, $r_s \in R_s$
and $x \in C_R(R_e)$.
Indeed, $\sigma_s(x) = \sum_{i=1}^n a_i x b_i$, for $x \in R_e$,
where $a_i \in R_s$ and $b_i \in R_{s^{-1}}$ are chosen so
that $\sum_{i = 1}^n a_i b_i = 1_R$.
If $H \subseteq G$ and $Y \subseteq C_R(R_e)$, then we
let $Y^H$ denote the set of $y \in Y$ which
are fixed by all $\sigma_s$, for $s \in H$.

\begin{theorem}[Miyashita \cite{miy}]\label{commutanttheorem}
Let $R$ be a unital ring strongly graded by the group $G$.
If $H$ is a subgroup of $G$, then $C_R(R_H) = C_R(R_e)^H$.
\end{theorem}

In fact, Miyashita proves a more general statement
concerning $G$-actions on module endomorphisms
(see Theorems 2.12 and 2.13 in \cite{miy}).
For more details concerning this and related results,
see e.g. \cite[Section I.2]{CaenOyst}, \cite[Theorem (2.1)]{Dade},
\cite[Section 3.4]{nas2004} and \cite{oin10}.
For more details about group graded rings in general,
see e.g. \cite{nas} or \cite{nas2004}.

The purpose of this article is to
generalize Theorem \ref{commutanttheorem} from
groups to groupoids (see Theorem \ref{maintheorem}).
To be more precise, suppose that $G$ is a small category, that is such that $\mor(G)$ is a set.
The family of objects of $G$ is denoted by $\ob(G)$;
we will often identify an object in $G$ with
its associated identity morphism.
The family of morphisms in $G$ is denoted by $\mor(G)$;
by abuse of notation, we will often write $s \in G$
when we mean $s \in \mor(G)$.
The domain and codomain of a morphism $s$ in $G$ is denoted by
$d(s)$ and $c(s)$ respectively.
We let $G^{(2)}$ denote the collection of composable
pairs of morphisms in $G$, that is all $(s,t)$ in
$\mor(G) \times \mor(G)$ satisfying $d(s)=c(t)$.
For $e,f \in \ob(G)$, we let $G_{f,e}$ denote the collection of $s\in G$
with $c(s)=f$ and $d(s)=e$ and $G_e$ denotes the monoid $G_{e,e}$.
A category is called \emph{cancellative} (a \emph{groupoid})
if all its morphisms are both monomorphisms
and epimorphisms (isomorphisms).
A subcategory of a groupoid is said to be a
\emph{subgroupoid} if it is closed under inverses.
For more details concerning categories in general and
groupoids in particular, see e.g. \cite{mac} and \cite{hig} respectively.
Let $R$ be a ring.
We say that a set of additive subgroups,
$R_s$, for $s \in G$, of $R$ is a \emph{$G$-filter} in $R$ if for all $s ,t \in G$,
we have $R_s R_t \subseteq R_{st}$ if
$(s,t) \in G^{(2)}$ and
$R_s R_t = \{ 0 \}$ otherwise.
We say that a $G$-filter is \emph{strong}
if $R_s R_t = R_{st}$ for $(s,t) \in G^{(2)}$.
Furthermore, we say that the ring $R$ is \emph{graded} by the category $G$
if there is a $G$-filter,
$R_s$, for $s \in G$, in $R$ such that
$R = \bigoplus_{s \in G} R_s$.
If $R$ is graded by a strong $G$-filter,
then we say that it is \emph{strongly graded}.
Analogously to the group graded situation, if $H$ is a subcategory of $G$, then we let $R_H$ denote the
subring $\bigoplus_{s \in H} R_s$ of $R$.
We say that $R$ is \emph{locally unital} if for each $e\in \ob(G)$ the ring $R_e$ is unital, making every $R_s$,
for $s\in G$, a unital $R_{c(s)}$-$R_{d(s)}$-bimodule.
For more details concerning category graded rings,
see e.g. \cite{lu05}, \cite{lu06}, \cite{lu07} and \cite{oinlun08}.

In Section \ref{commutants}, we show that if
$R$ is a ring which is strongly graded by a groupoid $G$,
then for each $s \in G$ there is a ring isomorphism
$\sigma_s$ from $C_{R_{G_{d(s)}}}(R_{d(s)})$ to $C_{R_{G_{c(s)}}}(R_{c(s)})$
(see Definition \ref{action})
with properties similar to the ones in
the group case above (see Proposition \ref{functor}).
In the end of Section \ref{commutants}, we use
this fact to show the following result.

\begin{theorem}\label{maintheorem}
Let $R$ be a locally unital ring strongly graded by the groupoid $G$.
If $H$ is a subgroupoid of $G$, then $C_R(R_H)$
equals the set of elements of the form
$\sum_{e \in \ob(G)} x_e$ where
$x_e \in R_{G_e}$, for $e \in \ob(G) \setminus \ob(H)$,
$x_e \in C_{R_{G_e}}(R_e)$, for $e \in {\rm ob}(H)$,
and $\sigma_s(x_{d(s)}) = x_{c(s)}$, for $s \in H$.
\end{theorem}

There is a well-developed theory for invertible bimodules of \emph{unital} rings
(see \cite{CaenOyst}, \cite{LeBruynVanOystaeyenVanDenBergh} and \cite{miy}).
However, in order to be able to generalize this theory to locally unital groupoid graded rings,
and in particular in order to show Theorem \ref{maintheorem}, we need to extend
the theory slightly (see Section \ref{MiyashitaAction}).
Namely, given unital subrings $A$ and $B$ of a (not necessarily unital) ring $R$ we say
that a unital $A$-$B$-submodule $X$ of $R$ is invertible if there is a unital $B$-$A$-submodule
$X^{-1}$ of $R$ such that $XX^{-1}=A$ and $X^{-1}X=B$. The collection of invertible
submodules of $R$ forms a groupoid
(see Definition \ref{DefEtt} for the details).

In Section \ref{examples}, we illustrate
Theorem \ref{commutanttheorem} and Theorem \ref{maintheorem}
in two cases (see Example \ref{simpleexamples}).
To this end, we make an explicit construction (see Proposition \ref{categorygradedex})
of graded rings, which is inspired by \cite{das99}. A particular case of our construction
implies the following result.

\begin{theorem}\label{groupoidexample}
Given a finite groupoid $G$, equipped with a nonidentity
morphism $t : d(t) \rightarrow c(t)$,
there is a unital strongly $G$-graded ring $R$
with the properties that each $R_s$, for $s \in G$,
is nonzero and $R_t$ is nonfree as a left $R_{c(t)}$-module.
\end{theorem}

We find that Theorem \ref{groupoidexample} is interesting in its
own right since, in general, every component $R_s$, for $s \in G$,
of a strongly groupoid graded ring $R$, is finitely
generated and projective as a left $R_{c(s)}$-module
(see Proposition \ref{cancellable}(e)).

\section{Miyashita Action}\label{MiyashitaAction}

Throughout this section, let $A$, $B$, $C$, $R$ and $S$ be rings
such that $A$, $B$ and $C$ are unital subrings of $R$.
Furthermore, let $M$, $N$ and $P$ be $R$-$S$-bimodules;
we let $\Hom_{R,S}(M, N)$ denote the collection of
simultaneously left $R$-linear and right
$S$-linear maps $M \rightarrow N$.

\begin{definition}\label{DefEtt}
We say that a unital $A$-$B$-submodule $X$ of $R$ is \emph{invertible}
in $R$ if there is a unital $B$-$A$-submodule $X^{-1}$ of $R$
such that $X X^{-1} = A$ and $X^{-1} X = B$.
Let $\Grd(R)$ denote the groupoid having subrings of $R$
as objects and invertible $A$-$B$-submodules $X$ of $R$ as morphisms,
for subrings $A$ and $B$ of $R$;
in that case we will write $X : B \rightarrow A$.
If $Y : C \rightarrow B$ is an invertible $B$-$C$-submodule of $R$,
then the composition of $X$ and $Y$ is defined as
the $A$-$C$-submodule $XY$ of $R$.
The identity morphism $A \rightarrow A$ is $A$ itself.
\end{definition}

\begin{proposition}\label{projective}
Every $X : B \rightarrow A$ in $\Grd(R)$
is finitely generated and projective
both as a left $A$-module
and a right $B$-module.
\end{proposition}

\begin{proof}
By the assumptions $A = X X^{-1}$ and hence there
is a positive integer $n$ and $x_i \in X$
and $y_i \in X^{-1}$, for $i \in \{1,\ldots,n\}$,
such that $1_A = \sum_{i=1}^n x_i y_i$.
For each $i \in \{1,\ldots,n\}$ define a
right $B$-linear $f_i : X \rightarrow B$ by $f_i(x)= y_i x$, for $x \in X$.
If $x \in X$, then
$x = 1_A x = \sum_{i=1}^n  x_i y_i x = \sum_{i=1}^n x_i f_i(x)$.
Hence, by the dual basis lemma (see e.g \cite[p. 23]{Lam}),
we get that $X$ is a projective right $B$-module
generated by $x_1,\ldots,x_n$.
Analogously, one can prove that $X$ is a finitely
generated projective left $A$-module.
\end{proof}

\begin{proposition}\label{miyashita}
If $X : B \rightarrow A$ is in $\Grd(R)$ and
$f \in \Hom_{B,S}(BM, BN)$, then there is a unique
$f^X \in \Hom_{A,S}(AM, AN)$ satisfying
\begin{equation}\label{property}
f^X(xm) = xf(1_B m)
\end{equation}
for all $x \in X$ and all $m \in M$.
Moreover, the following properties hold:
\begin{enumerate}[{\rm (a)}]
\item $0^X = 0$ and ${\rm id}_{BM}^X = {\rm id}_{AM}$;

\item if $g \in \Hom_{B,S}(B M, B N)$, then
$(f+g)^X = f^X + g^X$;

\item if $g \in \Hom_{B,S}(B N, B P)$, then
$(g \circ f)^X = g^X \circ f^X$;

\item if $g \in \Hom_{A,S}(AM,AN)$, then $g^A = g$;

\item if $Y : C \rightarrow B$ in $\Grd(R)$
and $g \in \Hom_{C,S}(CM,CN)$, then $(g^Y)^X = g^{XY}$.
\end{enumerate}
\end{proposition}

\begin{proof}
Fix $X : B \rightarrow A$ in $\Grd(R)$ and $f \in \Hom_{B,S}(BM,BN)$.
Since $1_A \in A = XX^{-1}$, there is a positive integer $n$
and $x_i \in X$, $y_i \in X^{-1}$, for $i \in \{1,\ldots,n\}$,
such that $\sum_{i=1}^n x_i y_i = 1_A$.
If a map $f^X \in \Hom_{A,S}(AM,AN)$ satisfying (\ref{property}) exists,
then it is unique, since
$$f^X(a m) = f^X(a 1_A m) = f^X \left( a \sum_{i=1}^n x_i y_i m \right) =$$
\begin{equation}\label{define}
= \sum_{i=1}^n f^X(a x_i y_i m)  = a \sum_{i=1}^n x_i f(y_i m)
\end{equation}
for all $a \in A$ and all $m \in M$; define $f^X(a m)$ by the last part of (\ref{define}).
We must show that $f^X$ does not depend on the
choice of the $x_i$'s and $y_i$'s.
To this end, suppose that $p$ is a positive integer and
$x_j' \in X$ and $y_j' \in X^{-1}$, for $j \in \{1,\ldots,p\}$,
are chosen so that $\sum_{j=1}^p x_j' y_j' = 1_A$.
Take $a \in A$ and $m \in M$.
Then, since $y_i x_j' \in B$, we get that
$$a \sum_{j=1}^p x_j' f(y_j' m) =
a \sum_{j=1}^p 1_A x_j' f(y_j' m) =
a \sum_{j=1}^p \sum_{i=1}^n x_i y_i x_j' f(y_j' m) =$$
$$= a \sum_{j=1}^p \sum_{i=1}^n x_i f( y_i x_j' y_j' m)
= a \sum_{i=1}^n \sum_{j=1}^p x_i f( y_i x_j' y_j' m) =$$
$$= a \sum_{i=1}^n x_i f \left( y_i \sum_{j=1}^p x_j' y_j' m \right)
= a \sum_{i=1}^n x_i f(y_i 1_A m)
= a \sum_{i=1}^n x_i f(y_i m).$$

Now we show that (\ref{property}) holds.
If $x \in X$ and $m \in M$, then,
since $y_i x \in X^{-1}X = B$ for $i\in \{1,\ldots,n\}$, we get that
$$f^X(xm) = \sum_{i=1}^n x_i f(y_i x m) =
\sum_{i=1}^n x_i f(y_i x 1_B m) =$$
$$= \sum_{i=1}^n x_i y_i x f(1_B m)=
1_A xf(1_B m) = xf(1_B m).$$

Next we show that $f^X \in \Hom_{A,S}(AM, AN)$.
It is clear that $f^X$ respects addition and right $S$-multiplication.
Now we show that $f^X$ respects left $A$-multiplication.
To this end, suppose that
$m \in M$ and $a,a' \in A$.
Since $ax_i \in X$, for $i \in \{1,\ldots,n\}$, we get,
by (\ref{property}), that
$$f^X(aa'm) = f^X(a 1_A a' m ) =
f^X \left( a \sum_{i=1}^n x_i y_i a' m  \right) =
\sum_{i=1}^n f^X(a x_i y_i a' m ) =$$
$$= \sum_{i=1}^n a x_i f(1_B y_i a' m ) =
a \left( \sum_{i=1}^n x_i f(y_i a' m) \right)  =
a f^X(1_A a' m) = a f^X(a'm).$$

(a) and (b) follow immediately.

(c) It is clear that both $(g \circ f)^X$ and
$g^X \circ f^X$ belong to $\Hom_{A,S}(A M, A P)$.
Moreover, if $x \in X$ and $m \in M$, then we get that
$$ (g^X \circ f^X)(xm) = g^X(f^X(xm)) =
g^X(x f(1_B m)) = $$
$$= x g(f(1_B m)) = x (g \circ f)(1_B m).$$
By uniqueness of the map $h^X$ in $\Hom_{A,S}(A M, A P)$
satisfying $h^X(xm) = xh(1_B m)$, for $x \in X$
and $m \in M$, it follows that $(g \circ f)^X = g^X \circ f^X$.

(d) This follows if we let
$x_1=y_1=1_A$ and $x_i = y_i = 0$, for $i \in \{2,\ldots,n\}$.

(e) Suppose that $g \in \Hom_{C,S}(CM,CN)$
and that $Y : C \rightarrow B$. Take a positive
integer $p$ and $x_j' \in Y$, $y_j' \in Y^{-1}$,
for $j \in \{1,\ldots,p\}$, such that
$\sum_{j=1}^p x_j' y_j' = 1_B$.
If $a \in A$ and $m \in M$, then
$$(g^Y)^X (am) = a \sum_{i=1}^n x_i g^Y(y_i m) =
a \sum_{i=1}^n \sum_{j=1}^p x_i x_j' g(y_j' y_i m) = g^{XY}(am)$$
since for each $i$ and $j$ we have
$x_i x_j' \in XY$, $y_j' y_i \in Y^{-1}X^{-1} = (XY)^{-1}$ and
$$\sum_{i=1}^n \sum_{j=1}^p x_i x_j' y_j' y_i =
\sum_{i=1}^n x_i \left( \sum_{j=1}^p x_j' y_j' \right) y_i=
\sum_{i=1}^n x_i 1_B y_i = \sum_{i=1}^n x_i y_i = 1_A.$$
\end{proof}

\begin{definition}
Suppose that $G$ and $H$ are categories.
Recall that an action of $G$ on $H$
is a functor \ $\widehat \ : G \rightarrow H$.
If $H$ is a category of abelian categories,
then we say that an action \ $\widehat{}$ \ of $G$
on $H$ is additive if for each morphism $g$ in $G$,
the functor $\widehat{g}$ respects
the additive structures on the hom-sets.
\end{definition}

\begin{remark}\label{firstremark}
For each subring $A$ of $R$, we let $\Hom_{A,S}$ denote
the abelian category having $A$-$S$-bimodules $AM$ as objects,
for $R$-$S$-bimodules $M$, and $A$-$S$-bimodule maps
$f : AM \rightarrow AN$ as morphisms, for $R$-$S$-bimodules
$M$ and $N$. Furthermore, we let $\Hom_S$ denote the
category having $\Hom_{A,S}$ as objects, for subrings
$A$ of $R$, and functors $\Hom_{B,S} \rightarrow \Hom_{A,S}$
as morphisms, for subrings $A$ and $B$ of $R$.
Then {\rm Proposition \ref{miyashita}} can be formulated
by saying that there is a unique additive action \ $\widehat{}$ \
of $\Grd(R)$ on $\Hom_S$ subject to the condition that
for any $X : B \rightarrow A$ in $\Grd(R)$,
any $R$-$S$-bimodules $M$ and $N$,
and any $f \in \Hom_{B,S}(BM, BN)$, we have that
$\widehat{X}(f)(xm) = xf(1_B m)$
for all $x \in X$ and all $m \in M$.
\end{remark}

\begin{proposition}\label{ringproposition}
For any $X : B \rightarrow A$ in $\Grd(R)$
there is a unique ring isomorphism
$\sigma^X : C_{BR}(B) \rightarrow C_{AR}(A)$
with the property that
$\sigma^X(r) x = x r$, for $r \in C_{BR}(B)$
and $x \in X$.
If we choose a positive integer $n$ and
$x_i \in X$ and $y_i \in X^{-1}$, for
$i \in \{1 , \ldots , n\}$, satisfying
$\sum_{i=1}^n x_i y_i = 1_A$, then
$\sigma^X (r) = \sum_{i=1}^n x_i r y_i$,
for $r \in C_{BR}(B)$.
Moreover, $\sigma^A = {\rm id}_{C_{AR}(A)}$
and if $X : B \rightarrow A$ and
$Y : C \rightarrow B$ belong to $\Grd(R)$,
then $\sigma^{XY} = \sigma^X \circ \sigma^Y$.
\end{proposition}

\begin{proof}
For each subring $A$ of $R$, define maps
$h^A : \Endom_{A,R}(AR) \rightarrow C_{AR}(A)$
and
$h_A : C_{AR}(A) \rightarrow \Endom_{A,R}(AR)$
by $h^A(f) = f(1_A)$, for $f \in \Endom_{A,R}(AR)$,
respectively $h_A(c)(ar) = car$, for $c \in C_{AR}(A)$,
$a \in A$ and $r \in R$. It is clear that
$h^A$ and $h_A$ are well defined ring homomorphisms
satisfying $h^A \circ h_A = {\rm id}_{C_{AR}(A)}$
and $h_A \circ h^A = {\rm id}_{\Endom_{A,R}(AR)}$.
Suppose that $X : B \rightarrow A$ is in $\Grd(R)$ and that
there is a ring isomorphism
$\sigma^X : C_{BR}(B) \rightarrow C_{AR}(A)$
with the property that
$\sigma^X(r) x = x r$, for $r \in C_{BR}(B)$
and $x \in X$. By the above, it follows that
for each $f \in \Endom_{B,R}(BR)$ the map
$(h_A \circ \sigma^X \circ h^B)(f)
\in \Endom_{A,R}(AR)$ satisfies
$$(h_A \circ \sigma^X \circ h^B)(f)(xr) =
h_A (\sigma^X (h^B(f))) (xr) =$$
$$= \sigma^X(h^B(f)) xr =
x h^B(f) r = x f(1_B) r$$
for all $x \in X$ and all $r \in R$;
by uniqueness, we get that
$(h_A \circ \sigma^X \circ h^B)(f) = f^X$.
Hence, if $r \in C_{BR}(B)$, then we get that
$$\sigma^X(r) = (h^A \circ h_A \circ \sigma^X
\circ h^B \circ h_B)(r) =
(h^A \circ (\cdot)^X \circ h_B)(r) =$$
$$= h^A(h_B(r)^X) = h_B(r)^X(1_A) = \sum_{i=1}^n x_i r y_i.$$
By Proposition \ref{miyashita}(a)-(e), it follows
that $\sigma^X$ is a ring isomorphism
satisfying $\sigma^A = {\rm id}_{C_{AR}(A)}$
and $\sigma^{XY} = \sigma^X \circ \sigma^Y$.
\end{proof}

\begin{remark}\label{ringremark}
If we for each subring $A$ of $R$, consider
the ring $C_{AR}(A)$ to be an abelian category with
one object $AR$, then the disjoint union
$C(R) := \biguplus C_{AR}(A)$, where the union
runs over all subrings $A$ of $R$, has an induced
structure of an abelian category.
Therefore, {\rm Proposition \ref{ringproposition}}
can be formulated by saying that
the action of $\Grd(R)$ on $\Hom_S$
defined in {\rm Remark \ref{firstremark}} induces
a unique additive action \ $\widehat{}$ \ of $\Grd(R)$ on $C(R)$
subject to the condition that
for each $X : B \rightarrow A$ in $\Grd(R)$,
the equality
$\widehat{X}(r) x = x r$ holds
for all $r \in C_{BR}(B)$ and all $x \in X$.
\end{remark}

The commutant $C_A(A)$ is called the \emph{center}
of $A$ and is denoted by $Z(A)$.

\begin{proposition}\label{centerring}
For any $X : B \rightarrow A$ in $\Grd(R)$
there is a unique ring isomorphism
$\sigma^X : Z(B) \rightarrow Z(A)$
with the property that
$\sigma^X(r) x = x r$, for $r \in Z(B)$
and $x \in X$.
If we choose a positive integer $n$ and
$x_i \in X$ and $y_i \in X^{-1}$, for
$i \in \{1 , \ldots , n\}$, satisfying
$\sum_{i=1}^n x_i y_i = 1_A$, then
$\sigma^X (r) = \sum_{i=1}^n x_i r y_i$,
for $r \in Z(B)$.
Moreover, $\sigma^A = {\rm id}_{Z(A)}$
and if $X : B \rightarrow A$ and
$Y : C \rightarrow B$ belong to $\Grd(R)$,
then $\sigma^{XY} = \sigma^X \circ \sigma^Y$.
\end{proposition}

\begin{proof}
This follows immediately from
Proposition \ref{ringproposition}.
\end{proof}

\begin{remark}\label{centerremark}
If we for each subring $A$ of $R$, consider
the ring $Z(A)$ to be an abelian category with
one object $A$, then the disjoint union
$Z_R := \biguplus Z(A)$, where the union
runs over all subrings $A$ of $R$, has an induced
structure of an abelian category.
Therefore, {\rm Proposition \ref{centerring}}
can be formulated by saying that
the action of $\Grd(R)$ on $\Hom_S$
defined in {\rm Remark \ref{firstremark}} induces
a unique additive action \ $\widehat{}$ \ of $\Grd(R)$ on $Z_R$
subject to the condition that
for each $X : B \rightarrow A$ in $\Grd(R)$,
the equality $\widehat{X}(r) x = x r$ holds
for all $r \in Z(B)$ and all $x \in X$.
\end{remark}

\section{Graded Rings}\label{commutants}

At the end of this section,
we prove Theorem \ref{maintheorem}.
To achieve this, we first show
three propositions concerning rings graded by
categories and, in particular, groupoids.

\begin{proposition}\label{cancellable}
Let $R$ be a locally unital ring graded by a category $G$.
\begin{enumerate}[{\rm (a)}]
\item If $s \in G$ is an isomorphism, then
$R_s R_{s^{-1}} = R_{c(s)}$
if and only if
$R_s R_t = R_{st}$ for all $t \in G$
with $d(s)=c(t)$.
In particular, if $G$ is a groupoid (or group),
then $R$ is strongly graded if and only if
$R_s R_{s^{-1}} = R_{c(s)}$, for all $s \in G$.

\item Suppose that $R$ is strongly graded. If $s \in
G$ is an isomorphism, then
$R_s$ is finitely generated and projective,
both as a left $R_{c(s)}$-module and a right $R_{d(s)}$-module.
In particular, if $G$ is a groupoid then the same
conclusion holds for each $s \in G$.

\item The ring $R$ is unital if and only if $R=R_{H}=\bigoplus_{s\in H} R_s$ for
a subcategory $H$ of $G$ with finitely many objects.
The subcategory $H$ may be chosen so that $1_{R_e}$
is nonzero for all $e \in \ob(H)$.
\end{enumerate}
\end{proposition}

\begin{proof}
(a) The ''if'' statement is clear. Now we show the
''only if'' statement. Take $(s,t) \in G^{(2)}$ and
suppose that $R_s R_{s^{-1}} = R_{c(s)}$.
Then, by the assumptions we get that
$R_s R_t \subseteq R_{st} = R_{c(s)} R_{st} =
R_s R_{s^{-1}} R_{st} \subseteq R_s R_{s^{-1}st} =
R_s R_t$. Therefore, $R_s R_t = R_{st}$.
The last part follows immediately.

(b) This follows from Proposition \ref{projective}.

(c) The ''if'' statement is clear since $\ob(H)$ is finite,
then $\sum_{e\in\ob(H)} 1_{R_e}$ is an identity element of $R$.
Now we show the ''only if'' statement of the claim. Suppose
that $R$ has an identity element $1_R = \sum_{s\in G} r_s$ for
some $r_s \in R_s$, for $s\in G$, such that $r_s = 0$
for all but finitely many $s\in G$.
Take $e,f \in \ob(G)$. If $e\neq f$, then
$0 = 1_{R_e} 1_{R_f} = 1_{R_e} 1_R 1_{R_f} = \sum_{s\in G} 1_{R_e} r_s 1_{R_f} =
\sum_{s\in G_{e,f}} r_s$.
This implies that $r_s=0$ for all $s\in G$ with $d(s)\neq c(s)$.
Also $1_{R_e} = 1_{R_e} 1_{R_e} = 1_{R_e} 1_R 1_{R_e} = \sum_{s\in G} 1_{R_e} r_s 1_{R_e} = \sum_{s\in G_{e}} r_s$.
This implies that $r_e = 1_{R_e}$ and that $r_s = 0$ for all nonidentity $s\in G$ with $d(s)=c(s)$.
Therefore $1_R = \sum_{e\in \ob(G)} 1_{R_e}$ which in turn implies that $1_{R_e} = 0$ for all but
finitely many $e\in \ob(G)$. Put $H = \{s\in G \, \mid \, 1_{R_{d(s)}} = 1_{R_{c(s)}} \neq 0\}$.
Then $H$ is a finite object subcategory of $G$ satisfying $R = R_H$.
\end{proof}

In general there is not any obvious connection between local unitality and unitality of a graded ring. This is illustrated by the following remark.

\begin{remark}
{\rm (a)} If $R$ is a unital ring graded by a cancellative category, then $R$ is also a locally unital ring.
Indeed, let us write $1_R = \sum_{s\in G} 1_s$ where $1_s \in R_s$ for $s\in G$. If $t \in G$, then $1_t = 1_R 1_t = \sum_{s \in G} 1_s 1_t$.
Since $G$ is cancellative, this implies that $1_s 1_t = 0$ whenever $s \in G \setminus \ob(G)$.
Therefore, if $s \in G \setminus \ob(G)$, then $1_s = 1_s 1_R = \sum_{t \in G} 1_s 1_t = 0$.
It is clear that $\{ 1_e \}_{e \in \ob(G)}$ is a set of local units for $R$.

{\rm (b)} The conclusion in {\rm (a)} does not hold if $G$ is not cancellative. 
Indeed, let $G=\{e,s\}$ be the monoid with $e^2=e$, $s^2=s$ and $es=se=s$.
Define
\begin{displaymath}
R=
\left(\begin{array}{ccc}
	\C & \C & 0 \\
	\C & \C & 0 \\
	0 & 0 & \C \\
\end{array}\right)
\quad
R_e =	\left(\begin{array}{ccc}
	0 & 0 & 0 \\
	0 & 0 & 0 \\
	0 & 0 & \C \\
\end{array}\right)
\quad
R_s =
	\left(\begin{array}{ccc}
	\C & \C & 0 \\
	\C & \C & 0 \\
	0 & 0 & 0 \\
\end{array}\right).
\end{displaymath}
Then $R=R_e \bigoplus R_s$ is a unital $G$-graded ring which is not a locally unital ring.

{\rm (c)} There are examples of $G$-graded rings $R$ which are non-unital, but locally unital.
Indeed, suppose that $G$ is a category with $\ob(G)$ infinite
and that $K$ is a non-trivial ring which is unital. Let $R=KG$ be the category algebra
of $G$ over $K$ (this is sometimes called a quiver algebra
of $G$ over $K$, see e.g. \cite{DerksenWeyman}). 
Recall that $KG$ is the set of formal sums $\sum_{s\in G} k_s u_s$
where $k_s\in K$, for $s\in G$, and $k_s = 0$ for all but finitely many $s\in G$.
The addition on $KG$ is defined by $\sum_{s\in G} k_s u_s + \sum_{s\in G} k_s' u_s = \sum_{s\in G} (k_s + k_s') u_s$
and the multiplication is defined as the bilinear extension of the rule $(k_s u_s) (k'_t u_t) = k_s k'_t u_{st}$
for $s,t\in G$ and $k_s,k'_t\in K$ if $c(t)=d(s)$ and $(k_s u_s) (k'_t u_t) =0$ otherwise.
If we put $R_s = K u_s$, for $s\in G$, then $R=\bigoplus_{s\in G} R_s$ is a (strongly) $G$-graded ring.
For each $e\in \ob(G)$, it is clear that $R_e$ is a unital ring with identity $1_K u_e$.
This makes $R$ a locally unital ring. However, from {\rm Proposition \ref{cancellable}(c)}
and the fact that $\ob(G)$ is infinite, it follows that $R$ is non-unital.
\end{remark}

\begin{definition}\label{action}
Suppose that $R$ is a locally unital ring strongly graded by a
groupoid $G$.
By Proposition \ref{ringproposition} we can use the invertible $R_{c(s)}$-$R_{d(s)}$-bimodules $R_s$, for $s\in G$,
to define a subgroupoid $C(R,G)$ of $\Grd(R)$ with
$C_{R_{G_e}}(R_e)$,
for $e \in \ob(G)$, as objects,
and the ring isomorphisms \linebreak
$C_{R_{G_{d(s)}}}(R_{d(s)}) \rightarrow C_{R_{G_{c(s)}}}(R_{c(s)})$,
for $s \in G$, as morphisms. In the sequel, these will be denoted by $\sigma_s$.
\end{definition}

\begin{proposition}\label{functor}
Suppose that $R$ is a locally unital ring strongly graded by a
groupoid $G$. Then the association of each $e \in \ob(G)$
and each $s \in G$ to the ring $C_{R_{G_e}}(R_e)$
and the function
$\sigma_s : C_{R_{G_{d(s)}}}(R_{d(s)}) \rightarrow C_{R_{G_{c(s)}}}(R_{c(s)})$,
respectively,
defines a functor of groupoids $\sigma : G \rightarrow C(R,G)$.
Moreover, $\sigma$ is uniquely defined on morphisms
given that the relation $\sigma_s(x)r_s = r_s x$ holds
for all $s \in G$, all $x \in C_R(R_{d(s)})$ and all $r_s \in R_s$.
\end{proposition}

\begin{proof}
This follows immediately from Proposition \ref{ringproposition} (or Remark \ref{ringremark}).
\end{proof}

\begin{remark}
Suppose that $R$ is a locally unital ring strongly graded by a groupoid $G$.
Take $s \in G$.
By {\rm Proposition \ref{cancellable}(a)} and the
equalities $R_{c(s)} = R_s R_{s^{-1}}$
and $R_{d(s)} = R_{s^{-1}} R_s$ it follows that
$R_{d(s)}=0$ if and only if $R_{c(s)}=0$;
in that case $\sigma_s$ is of course the zero map.
If one wants to avoid such maps one may,
by {\rm Proposition \ref{cancellable}(c)},
assume that all components of $R$ are nonzero
and in particular that each ring $R_e$, for $e \in \ob(G)$,
has a nonzero identity element.
\end{remark}

\begin{definition}
Suppose that $R$ is a locally unital ring strongly graded by a
groupoid $G$. By abuse of notation,
we let $Z(R,G)$ denote the subcategory of $C(R,G)$
having $Z(R_e)$, for $e \in \ob(G)$, as objects,
and the ring isomorphisms
$Z(R_{d(s)}) \rightarrow Z(R_{c(s)})$,
for $s \in G$, as morphisms.
\end{definition}

\begin{proposition}\label{functorcenter}
Suppose that $R$ is a locally unital ring strongly graded by a
groupoid $G$. Then the association of each $e \in \ob(G)$
and each $s \in G$ to the ring
$Z(R_e)$ and the function
$\sigma_s : Z(R_{d(s)}) \rightarrow Z(R_{c(s)})$, respectively,
defines a functor of groupoids
$\sigma : G \rightarrow Z(R,G)$.
Moreover, $\sigma$ is uniquely defined on morphisms
given that the relation $\sigma_s(x)r_s = r_s x$ holds
for all $s \in G$, all $x \in Z(R_{d(s)})$ and all $r_s \in R_s$.
\end{proposition}

\begin{proof}
This follows immediately from Proposition \ref{centerring} 
(or Remark \ref{centerremark}).
\end{proof}

\noindent {\bf Proof of Theorem \ref{maintheorem}.}
Suppose that $y = \sum_{s \in G} y_s \in C_R(R_H)$
where $y_s \in R_s$, for $s \in G$, and
$y_s = 0$ for all but finitely many $s \in G$.
Since $1_e y = y 1_e$, for $e \in \ob(G)$, we get that
$y_s = 0$ whenever $c(s) \neq d(s)$.
Therefore, we get that $y = \sum_{e \in \ob(G)} x_e$, where
$x_e := \sum_{s \in G_e} y_s \in R_{G_e}$, for $e \in \ob(G)$.
Since $y \in C_R(R_H) \subseteq C_R(R_e)$, for $e \in \ob(H)$,
we get that $x_e \in C_{R_{G_e}}(R_e)$, for $e \in \ob(H)$.
Take $s \in H$. By the last part of Proposition \ref{functor}
and the fact that the equality
$r_s y = y r_s$ holds for all $r_s \in R_s$, we get that
$\sigma_s(x_{d(s)}) = x_{c(s)}$.
On the other hand, it is clear, by Proposition \ref{functor}, that
all sums of the form $\sum_{e \in \ob(G)} x_e$,
with $x_e \in R_{G_e}$, when $e \in \ob(G) \setminus \ob(H)$,
and $x_e \in C_{R_{G_e}}(R_e)$, for $e \in \ob(H)$, satisfying
$\sigma_s(x_{d(s)}) = x_{c(s)}$, for $s \in H$,
belong to $C_R(R_H)$. {\hfill $\square$}

\section{Examples}\label{examples}

In this section, we show Theorem \ref{groupoidexample} and
illustrate it in two cases (see Example \ref{simpleexamples}). Our
method will be to generalize, to category graded rings (see
Proposition \ref{categorygradedex}), the construction given in
\cite{das99} for the group graded situation. In order to do this, we first
need to introduce some additional notation. Let $K$ be a commutative ring
with $1_K \neq 0$ and suppose that $G$ is a category. Fix a positive integer $n$ and
choose $s_i \in G$, for $1 \leq i \leq n$. Put $S = \{ s_i \mid 1
\leq i \leq n \}$. If $1 \leq i,j \leq n$, then let $e_{ij} \in M_n(K)$
be the matrix with $1_K$ in the $ij$:th position and 0 elsewhere. For $s
\in G$, we let $R_s$ be the left $K$-submodule of $M_n(K)$ spanned
by the set
$\{ e_{ij} \mid 1 \leq i,j \leq n, \ (s_i,s) \in G^{(2)}, \ s_i s = s_j \}.$
With the above notation, the following result holds.

\begin{proposition}\label{categorygradedex}
If we put $R := \sum_{s \in G} R_s$, then
\begin{enumerate}[{\rm (a)}]
\item the collection of left $K$-modules $R_s$, for $s \in G$,
of $R$ is a $G$-filter in $R$;

\item if $s_i s \in S$, for all $(s_i,s) \in (S \times G) \cap G^{(2)}$, then $R_s$, for $s \in G$, is a strong
$G$-filter in $R$;

\item if $G = S$, then $R_s \neq \{ 0 \}$ for $s \in G$;

\item if $d(s_i) \in S$, for $i\in \{1,\ldots,n\}$,
then $R$ has an identity element given by $\sum_{f \in \ob(G)} 1_f$,
where for each $f \in \ob(G)$, the element $1_f \in R_f$ is the sum
of all $e_{ii}$ satisfying $d(s_i)=f$;

\item if $G$ is cancellative, then the collection
of left $K$-modules $R_s$, for $s \in G$, of $R$ makes $R$ a graded
ring;

\item if $G$ is a groupoid and $G = S$, then $R_s$, for $s \in G$,
makes $R$ a unital strongly graded ring with $R_s \neq \{ 0 \}$, for $s \in G$.
\end{enumerate}
\end{proposition}

\begin{proof}
(a) Suppose that $(s,t) \in G^{(2)}$. Take $e_{ij} \in R_s$ and
$e_{lk} \in R_t$. If $j \neq l$, then $e_{ij} e_{lk} = 0 \in
R_{st}$. Now let $j = l$. Then, since $s_i s = s_j$ and $s_j t =
s_k$, we get that $s_i st = s_j t = s_k$. Hence, $e_{ij} e_{jk} =
e_{ik} \in R_{st}$.

(b) Take $(s,t) \in G^{(2)}$ and $e_{ik} \in R_{st}$.
Then $s_i st = s_k$. Since $s_i s \in S$ there is
$s_j \in S$ with $s_i s = s_j$. This means that $e_{ij} \in R_s$.
Moreover, $s_j t = s_i s t = s_k$ which yields $e_{jk} \in R_t$.
Hence $e_{ik} = e_{ij} e_{jk} \in R_s R_t$.

(c) Take $s \in G$. Since $G = S$, there is $s_i,s_j \in S$
with $s_i = c(s)$ and $s_j = s$. Therefore $s_i s = c(s)s = s = s_j$.
Hence $e_{ij} \in S$ which, in turn, implies that $R_s \neq \{ 0 \}$.

(d) Take $s \in G$ and suppose that $e_{jk} \in R_s$ for some $j,k
\in \{ 1,\ldots,n \}$. By the assumptions we get that $d(s_j) \in
S$. Therefore $e_{jj} \in R_{d(s_j)}$ and hence $\sum_{f \in \ob(G)}
1_f e_{jk} = e_{jj} e_{jk} = e_{jk}$. In the same way $\sum_{f \in
\ob(G)} e_{jk} 1_f = e_{jk}$.

(e) Let $X_s$ denote the collection of pairs $(i,j)$, where $1 \leq
i,j \leq n$, such that $(s_i,s) \in G^{(2)}$ and $s_i s = s_j$.
Suppose that $s \neq t$. Seeking a contradiction suppose that $X_s
\cap X_t \neq \emptyset$. Then there are integers $k$ and $l$,
with $1 \leq k,l \leq n$, such that $s_k s = s_l = s_k t$. By the
cancellability of $G$ this implies that $s = t$ which is
a contradiction. Therefore, the
sets $X_s$, for $s \in G$, are pairwise disjoint. The claim now
follows from (a) and the fact that
$R_s = \sum_{(i,j) \in X_s} Ke_{ij}$ for all $s \in G$.

(f) This follows immediately from (a), (b), (c), (d) and (e).
If we use Proposition \ref{cancellable}(a) the strongness
condition can be proven directly in the following way.
Take $s \in G$ and $s_i \in S$. Since $G=S$ there is
$s_j \in S$ with $s_i s = s_j$. This means that $e_{ij} \in R_s$.
Since $G$ is a groupoid we get that $s_j s^{-1} = s_i$,
i.e. $e_{ji} \in R_{s^{-1}}$. Therefore
$e_{ii} = e_{ij} e_{ji} \in R_s R_{s^{-1}}$.
\end{proof}

\noindent {\bf Proof of Theorem \ref{groupoidexample}.}
We first consider the case when $G$ is connected.
If $G$ only has one object, then it is a group in which case
it has already been treated in \cite{das99}.
Therefore, from now on, we assume that we can choose
two different objects $e$ and $f$ from $G$.
We denote the morphisms of $G$
by $t_1,t_2,\ldots,t_n$. For technical reasons,
we suppose that $d(t_1) = f$, $c(t_1) = e$ and $t_n = e$.
Let us now choose $n+1$ morphisms $s_1,s_2,\ldots,s_{n+1}$ from $G$
in the following way; $s_i = t_i$, when $1 \leq i \leq n$, and $s_{n+1} = t_n$.
Now we define $R$ according to the beginning of this section.
By Proposition \ref{categorygradedex}(f), the ring $R$ is strongly $G$-graded
and each $R_s$, for $s \in G$, is nonzero.

We shall now show that the morphism $t := t_1$
has the desired property.
Let $m$ denote the cardinality of the set of
$s \in G$ with $d(g)=e$.
The component $R_e$ is the left $K$-module
spanned by the collection of $e_{ij}$ with $s_i e = s_j$,
that is, such that $s_i = s_j$ and $d(s_j)=e$.
By the construction of $S$ it follows that
the $K$-dimension of $R_e$ equals $m+3$.
Analogously, the component $R_{t_1}$ is the
left $K$-module spanned by the collection of $e_{ij}$
with $s_i t_1 = s_j$. Since $d(t_1) = f \neq e$,
this implies that the $K$-dimension of $R_{t_1}$
equals $m+1$.
Seeking a contradiction, suppose that $R_{t_1}$ is free on some generators
$u_l$, $1 \leq l \leq d$, as a left $R_e$-module.
Then the map $\theta : R_e^d \rightarrow
R_{t_1}$, defined by $\theta(x_1,\ldots,x_d) =
\sum_{l=1}^d x_l u_l$, for $x_l \in R_e$, for $l \in \{1,\ldots,d\}$, is,
in particular, an isomorphism of left $K$-modules.
Since ${\rm dim}_K(R_e^d) = d(m+3) > m+1 = {\rm dim}_K(R_{t_1})$,
this is impossible.

We shall now show that our groupoid $G$,
in the general case, is the disjoint union of connected groupoids.
Define an equivalence relation $\sim$ on $\ob(G)$ by saying
that $e \sim f$, for $e,f \in \ob(G)$, if there is a morphism
in $G$ from $e$ to $f$. Choose a set $E$ of representatives
for the different equivalence classes defined by $\sim$.
For each $e \in E$, let $[e]$ denote the equivalence class
to which $e$ belongs. Let $G_{[e]}$ denote the subgroupoid
of $G$ with $[e]$ as set of objects and morphisms $s \in G$
with the property that $c(s),d(s) \in [e]$.
Then each $G_{[e]}$, for $e \in E$, is a connected groupoid
and $G = \biguplus_{e \in E} G_{[e]}$.

For each $e \in E$, we now wish to define
a strongly $G_{[e]}$-graded ring $R_{[e]}$.
We consider three cases.
If $G_{[e]} = \{ e \}$, then let $R_{[e]} = K$.
If $[e] = \{ e \}$ but the group $G_{[e]}$ contains
a nonidentity morphism $t$, then let $R_{[e]}$
be any strongly $G_{[e]}$-graded ring with
the desired property (following \cite{das99}).
If $[e]$ has more than one element, let $R_{[e]}$
denote the strongly $G_{[e]}$-graded ring constructed
in the first part of the proof.
We may define a new ring to be the direct sum $\bigoplus_{e \in E} R_{[e]}$
which is strongly graded by $G$ and has the desired property.
{\hfill $\square$}

\begin{example}\label{simpleexamples}
We have chosen nontrivial examples of graded
rings $R$ in the sense
that not all graded components $R_s$ are free
left $R_{c(s)}$-modules.
In the free case the groupoid action is defined
by a single conjugation which makes the analysis
easier; in the general
case the action is a sum of such maps.

{\rm (a)} Suppose that $G$ is the cyclic additive
group ${\Bbb Z}_4 = \{ 0,1,2,3 \}$.
Using the notation from the proof of
{\rm Theorem \ref{groupoidexample}} above,
we put
$$s_1 = 0 \quad s_2 = 1 \quad s_3 = 2 \quad s_4 = s_5 = 3$$
Then $R:= M_5(K)$ is a strongly
${\Bbb Z}_4$-graded ring with components defined by
$$
\begin{array}{lcl}
R_0 &=& Ke_{11} + Ke_{22} + Ke_{33} + Ke_{44} + Ke_{45} +
Ke_{54} + Ke_{55} \\
R_1 &=& Ke_{12} + Ke_{23} + Ke_{34} + Ke_{35} + Ke_{41} + Ke_{51} \\
R_2 &=& Ke_{13} + Ke_{24} + Ke_{25} + Ke_{31} + Ke_{42} + Ke_{52} \\
R_3 &=& Ke_{14} + Ke_{15} + Ke_{21} + Ke_{32} + Ke_{43} + Ke_{53}
\end{array}
$$
By a straightforward calculation, we get that
$$
C_R(R_0) = Ke_{11} + Ke_{22} + Ke_{33} + K(e_{44}+e_{55}).
$$
It is easy to see that
$$\sigma_1(x) = e_{12}xe_{21} + e_{23}xe_{32} +
e_{34}xe_{43} + e_{41}xe_{14} + e_{51}xe_{15}$$
and hence that
$$\sigma_2(x) = \sigma_1^2(x) = e_{13}xe_{31} + e_{24}xe_{42} + e_{31}xe_{13}
+e_{42}xe_{24} + e_{52}xe_{25}$$
for all $x \in C_R(R_0)$.
If we put $H = \{ 0,2 \}$, then,
by {\rm Theorem \ref{commutanttheorem}}, we get that
$$ C_R(R_H) = C_R(R_0)^H = C_R(R_0)^{ \{ 2 \} } = K(e_{11} + e_{33}) +
K(e_{22}+e_{44} + e_{55})$$
and
$$Z(R) = C_R(R) = C_R(R_0)^{{\Bbb Z}_4} =
C_R(R_0)^{ \{ 1 \} } = K1_R.$$

{\rm (b)} Now suppose that $G$ is the groupoid with
two objects $e$ and $f$
and nonidentity morphisms $\alpha : e \rightarrow e$,
$\beta : f \rightarrow f$, $u_0 : f \rightarrow e$,
$u_1 : f \rightarrow e$, $t_0 : e \rightarrow f$
and $t_1 : e \rightarrow f$ with composition given by
the following relations
$$\alpha^2 = e \quad \alpha u_0 = u_1 \quad \alpha u_1 = u_0
\quad u_0 \beta = u_1 \quad u_1 \beta = u_0$$
$$\beta^2 = f \quad \beta t_0 = t_1 \quad \beta t_1 = t_0
\quad t_0 \alpha = t_1 \quad t_1 \alpha = t_0$$
$$u_0 t_0 = e \quad u_1 t_0 = \alpha \quad
u_0 t_1 = \alpha \quad u_1 t_1 = e$$
$$t_0 u_0 = f \quad t_0 u_1 = \beta \quad t_1 u_0 = \beta \quad t_1 u_1 = f$$
Using the notation from the proof of
{\rm Theorem \ref{groupoidexample}} above, we put
$$s_1 = f \quad s_2 = \beta \quad s_3 = u_0 \quad s_4 = u_1$$
$$s_5 = t_0 \quad s_6 = t_1 \quad s_7 = \alpha \quad s_8 = s_9 = e$$
Now we define the strongly $G$-graded subring $R$ of $M_9(K)$
according to the beginning of this section.
A straightforward calculation shows that
$$
\begin{array}{lcl}
R_e        &=& Ke_{55} + Ke_{66} + Ke_{77} + Ke_{88} + Ke_{89} + Ke_{98} + Ke_{99} \\
R_{\alpha} &=& Ke_{56} + Ke_{65} + Ke_{78} + Ke_{79} + Ke_{87} + Ke_{97} \\
R_{t_0}    &=& Ke_{15} + Ke_{26} + Ke_{38} + Ke_{39} + Ke_{47} \\
R_{t_1}    &=& Ke_{16} + Ke_{25} + Ke_{37} + Ke_{48} + Ke_{49} \\
R_f        &=& Ke_{11} + Ke_{22} + Ke_{33} + Ke_{44} \\
R_{\beta}  &=& Ke_{12} + Ke_{21} + Ke_{34} + Ke_{43} \\
R_{u_0}    &=& Ke_{51} + Ke_{62} + Ke_{74} + Ke_{83} + Ke_{93} \\
R_{u_1}    &=& Ke_{52} + Ke_{61} + Ke_{73} + Ke_{84} + Ke_{94}
\end{array}
$$
By a straightforward calculation we get that
$$C_{R_{G_e}}(R_e) = Ke_{55}+Ke_{66} + Ke_{77}+K(e_{88}+e_{99})$$
and
$$C_{R_{G_f}}(R_f) = Ke_{11}+Ke_{22} + Ke_{33}+Ke_{44}.$$
It is easy to see that
$$\sigma_{\alpha}(x) = e_{56}xe_{65} + e_{65}xe_{56} +
e_{78}xe_{87} + e_{87}xe_{78} + e_{97}xe_{79}$$
for all $x \in C_{R_{G_e}}(R_e)$
and that
$$\sigma_{\beta}(y) = e_{12}xe_{21} + e_{21}xe_{12} +
e_{34}xe_{43} + e_{43}xe_{34}$$
for all $y \in C_{R_{G_f}}(R_f)$.
Now we use this and {\rm Theorem \ref{maintheorem}} to compute $C_R(R_H)$
for all eleven subgroupoids $H$ of $G$:
$$H_1 = \{ e \} \quad H_2 = \{ f \} \quad H_3 = \{ e,f \} \quad
H_4 = G_e \quad H_5 = G_f$$
$$H_6 = \{ e,f,\alpha \} \quad H_7 = \{ e,f,\beta \} \quad H_8 = \{ e,f,\alpha,\beta \}$$
$$H_9 = \{ e,f,t_0,u_0 \} \quad H_{10} = \{ e,f,t_1,u_1 \} \quad H_{11} = G$$
We immediately get that
$$C_R(R_{H_1}) = C_R(R_e) = C_{R_{G_e}}(R_e) + R_{G_f} =$$
$$= Ke_{55} + Ke_{66}+ Ke_{77} + K(e_{88}+e_{99}) +
Ke_{11} + Ke_{22} + Ke_{33} + Ke_{44} + $$
$$+ Ke_{12} + Ke_{21} + Ke_{34} + Ke_{43}$$
and similarly that
$$C_R(R_{H_2}) = C_R(R_f) = C_{R_{G_f}}(R_f) + R_{G_e} =$$
$$= Ke_{11} + Ke_{22} + Ke_{33} + Ke_{44} +
Ke_{55} + Ke_{66} + Ke_{77} + Ke_{88} + $$
$$+ Ke_{89} + Ke_{98} + Ke_{99} + Ke_{56} + Ke_{65} +
Ke_{78} + Ke_{79} + Ke_{87} + Ke_{97}.$$
Furthermore, we get that
$$C_R(R_{H_3}) = C_{R_{G_f}}(R_f) + C_{R_{G_e}}(R_e) = $$
$$= Ke_{11} + Ke_{22} + Ke_{33} + Ke_{44} +
Ke_{55} + Ke_{66} + Ke_{77} + K(e_{88}+e_{99}).$$
Next we get that
$$C_R(R_{H_4}) = C_{R_{G_e}}(R_e)^{G_e} + R_{G_f} =
C_{R_{G_e}}(R_e)^{ \{ \alpha \} } + R_{G_f} =$$
$$= ( Ke_{55}+Ke_{66} + Ke_{77}+
K(e_{88}+e_{99}) )^{ \{ \alpha \} } + R_{G_f} =$$
$$= K(e_{55}+e_{66}) + K(e_{77}+e_{88}+e_{99}) + $$
$$+Ke_{11} + Ke_{22} + Ke_{33} + Ke_{44} + Ke_{12} + Ke_{21} + Ke_{34} + Ke_{43}$$
and
$$C_R(R_{H_5}) = C_{R_{G_f}}(R_f)^{G_f} + R_{G_e}=$$
$$= (Ke_{11}+Ke_{22} + Ke_{33}+Ke_{44})^{ \{ \beta \} } + R_{G_e} =$$
$$= K(e_{11}+e_{22}) + K(e_{33}+e_{44}) +
Ke_{55} + Ke_{66} + Ke_{77} + Ke_{88} + Ke_{89} + Ke_{98} + Ke_{99}$$
$$+ Ke_{56} + Ke_{65} + Ke_{78} + Ke_{79} + Ke_{87} + Ke_{97}.$$
By the above calculations, we get that
$$C_R(R_{H_6}) = C_{R_{G_e}}(R_e)^{G_e} + C_{R_{G_f}}(R_f) =$$
$$= K(e_{55}+e_{66}) + K(e_{77}+e_{88}+e_{99}) +
Ke_{11}+Ke_{22} + Ke_{33}+Ke_{44}$$
and
$$C_R(R_{H_7}) = C_{R_{G_f}}(R_f)^{G_f} + C_{R_{G_e}}(R_e) =$$
$$= K(e_{11}+e_{22}) + K(e_{33}+e_{44}) +
Ke_{55}+Ke_{66} + Ke_{77} + K(e_{88} + e_{99})$$
and
$$C_R(R_{H_8}) = C_{R_{G_e}}(R_e)^{G_e} + C_{R_{G_f}}(R_f)^{G_f} =$$
$$= K(e_{55}+e_{66}) + K(e_{77}+e_{88}+e_{99}) +
K(e_{11}+e_{22}) + K(e_{33}+e_{44}).$$
By a straightforward calculation, we get that
$$\sigma_{t_0}(x) = e_{15}xe_{51} + e_{26}xe_{62}+
e_{39}xe_{93} + e_{47}xe_{74}$$
and
$$\sigma_{t_1}(x) = e_{16}xe_{61} + e_{25}xe_{52}+
e_{37}xe_{73} + e_{48}xe_{84}$$
for all $x \in C_{R_{G_e}}(R_e)$.
By the above calculations, we get that
$$C_R(R_{H_9}) = \{ x + \sigma_{t_0}(x) \mid x \in C_{R_{G_e}}(R_e) \} =$$
$$= K(e_{11}+e_{55}) + K(e_{22}+e_{66}) + K(e_{44}+e_{77})+
K(e_{33}+e_{88}+e_{99})$$
and
$$C_R(R_{H_{10}}) = \{ x + \sigma_{t_1}(x) \mid x \in C_{R_{G_e}}(R_e) \} =$$
$$= K(e_{22}+e_{55}) + K(e_{11}+e_{66}) + K(e_{33}+e_{77})+
K(e_{44}+e_{88}+e_{99})$$
and
$$C_R(R_{H_{11}}) = Z(R) =
\{ x + \sigma_{t_0}(x) \mid x \in C_{R_{G_e}}(R_e)^{G_e} \} =$$
$$= K(e_{11} + e_{22} + e_{55} + e_{66}) +
K(e_{33} + e_{44} + e_{77} + e_{88} + e_{99}).$$
\end{example}


\end{document}